\documentclass[11pt,a4paper]{article}

\usepackage{amsmath,amsfonts,amsthm,amssymb}
\usepackage[english]{babel}
\usepackage{graphicx}
\usepackage{array}
\usepackage{tikz}
\usetikzlibrary{arrows}
\usetikzlibrary{positioning}

\usepackage[a4paper, total={6in, 9in}]{geometry}

\newtheorem{theorem}{Theorem}
\newtheorem{lemma}[theorem]{Lemma}
\newtheorem{conjecture}[theorem]{Conjecture}

\newtheorem{definition}[theorem]{Definition}
\newtheorem{proposition}[theorem]{Proposition}

\renewenvironment{proof}[1][\proofname]{{\noindent \bfseries #1}}{\qed}

\newcommand{\PP}{\mathbb{P}}
\newcommand{\EE}{\mathbb{E}}

\begin{document}

\title{A Proof of the Bar\'at-Thomassen Conjecture\thanks{The first and fourth authors were supported by ERC Advanced Grant GRACOL, project no. 320812. The second
author was supported by an FQRNT postdoctoral research grant and CIMI research fellowship. The fifth author was partially supported by the 
ANR Project STINT under Contract ANR-13-BS02-0007. \newline \indent Email addresses: \url{julien.bensmail.phd@gmail.com} (J. Bensmail), \url{ararat.harutyunyan@math.univ-toulouse.fr} (A. Harutyunyan), \url{tien-nam.le@ens-lyon.fr} (T.-N. Le), \url{martin.merker@uni-hamburg.de} (M. Merker), \url{stephan.thomasse@ens-lyon.fr} (S. Thomass\'e).}}

\author{Julien Bensmail$^a$, Ararat Harutyunyan$^b$, Tien-Nam Le$^c$, \\ Martin Merker$^a$, and St\'{e}phan Thomass\'{e}$^c$\\~\\
			\small $^a$Department of Applied Mathematics and Computer Science \\ \small Technical University of Denmark \\ \small DK-2800 Lyngby, Denmark\\~\\
			\small $^b$Institut de Math\'ematiques de Toulouse \\ \small Universit\'e de Toulouse Paul Sabatier \\ \small 31062 Toulouse Cedex 09, France\\~\\
			\small $^c$Laboratoire d'Informatique du Parall\'elisme \\ \small \'Ecole Normale Sup\'erieure de Lyon \\ \small 69364 Lyon Cedex 07, France}

\date{}

\maketitle

\begin{abstract}
The Bar\'at-Thomassen conjecture asserts that for every tree $T$ on $m$ edges, there exists a constant $k_T$ such that
every $k_T$-edge-connected graph with size divisible by $m$ can be edge-decomposed into copies of $T$.  
So far this conjecture has only been verified when $T$ is a path or when $T$ has diameter at most 4. 
Here we prove the full statement of the conjecture.
\end{abstract}


\section{Introduction} \label{section:introduction}

Unless stated otherwise, all graphs considered in this paper are finite and simple. 
Given a graph $G$, we denote by $V(G)$ and $E(G)$ its vertex and edge sets, respectively. 
For any subset $S$ of vertices or edges of $G$, we denote by $G[S]$ the subgraph of $G$ induced by $S$.

Given a connected graph $H$, we say that $G$ has an $H$-\textit{decomposition} if there exists a partition of $E(G)$ into parts isomorphic to $H$. 
When does a graph $G$ admit an $H$-decomposition? In general, this question is \textsf{NP}-complete whenever $H$ contains at least 3 edges, see~\cite{DT97}.
We focus on the case where we want to decompose $G$ into copies of a given tree $T$. 
A necessary condition for the existence of a $T$-decomposition is of course that $|E(T)|$ divides $|E(G)|$. 
There are many theorems and conjectures in graph theory stating that this condition is also sufficient in certain cases. 
By a result of Wilson~\cite{W76} this holds when $G$ is a sufficiently large complete graph, and there exist more general results showing that this is also true for graphs of large minimum degree. 
More precisely, for every tree $T$ there exists a constant $\varepsilon_T >0$ such that every graph $G$ of minimum degree $(1-\varepsilon_T )|V(G)|$ admits a $T$-decomposition, provided its size is divisible by the size of $T$ (see for example~\cite{BKLO16}).

A different line of research was started by Bar\'at and Thomassen~\cite{BT06}, when they observed in 2006 that $T$-decompositions are intimately related to nowhere-zero flows. Tutte conjectured that every 4-edge-connected graph admits a nowhere-zero 3-flow, but until recently it was not even known that any constant edge-connectivity suffices for this. 
Bar\'at and Thomassen showed that if every 8-edge-connected graph of size divisible by 3 admits a $K_{1,3}$-decomposition, then every 8-edge-connected graph admits a nowhere-zero 3-flow. 
Vice versa, they also showed that Tutte's 3-flow conjecture would imply that every 10-edge-connected graph with size divisible by 3 admits a $K_{1,3}$-decomposition. 
Motivated by this intrinsic connection, they conjectured the following.

\begin{conjecture}[Bar\'at-Thomassen Conjecture, \cite{BT06}] \label{conjecture:barat-thomassen}
For any tree $T$ on $m$ edges, there exists an integer $k_T$ such that every 
$k_T$-edge-connected graph with size divisible by $m$ has a $T$-decomposition.
\end{conjecture}

When the conjecture was made, it was only known to hold in the trivial cases where $T$ has less than 3 edges. 
Since then, Conjecture~\ref{conjecture:barat-thomassen} has attracted growing attention, and it has now been verified for different families of trees such as stars~\cite{Tho12}, some bistars~\cite{BG14,Thb13}, and paths of a certain length~\cite{BMOW15+,Tha08,Thb08,Tha13}.
Very recently, breakthrough results were obtained by Merker~\cite{Mer15+}, who proved the conjecture for all trees of diameter at most~$4$, hence covering some of the results above, and by Botler, Mota, Oshiro, and Wakabayashi~\cite{BMOW14}, who proved the conjecture for all paths.
The latter result was improved by Bensmail, Harutyunyan, Le, and Thomass\'e~\cite{BHLT15+}, who showed that, for path-decompositions, large minimum degree is a sufficient condition provided the graph is 24-edge-connected.

The purpose of this paper is to verify Conjecture~\ref{conjecture:barat-thomassen} for every tree $T$, hence settling the conjecture in the affirmative. 

\begin{theorem} \label{theorem:result}
The Bar\'at-Thomassen conjecture is true.
\end{theorem}

This paper builds upon previous work on the Bar\'at-Thomassen conjecture. It was shown by Thomassen in~\cite{Thb13}, and independently by Bar\'at and Gerbner in~\cite{BG14}, that it is sufficient to verify Conjecture~\ref{conjecture:barat-thomassen} for bipartite graphs $G$.

\begin{theorem}\emph{\cite{BG14,Thb13}}\label{theorem:bipartite}
Let $T$ be a tree on $m$ edges. The following two statements are equivalent:
\begin{enumerate}
	\item[(1)] There exists a natural number $k_T$ such that every $k_T$-edge-connected graph with size divisible by $m$ has a $T$-decomposition.
	\item[(2)] There exists a natural number $k'_T$ such that every $k'_T$-edge-connected bipartite graph with size divisible by $m$ has a $T$-decomposition.
\end{enumerate}
\end{theorem}

An important tool in the study of Conjecture~\ref{conjecture:barat-thomassen} is a recent result on modulo-$k$ orientations, which was proved by Thomassen in~\cite{Tho12}. 
He showed that the edges of a highly edge-connected graph can be oriented so that any prescribed out-degrees modulo $k$ are realized.
This was formerly known as Jaeger's weak $k$-flow conjecture and it immediately implies that Conjecture~\ref{conjecture:barat-thomassen} holds for all stars. 
Moreover, his result implies that every 8-edge-connected graph admits a nowhere-zero 3-flow. In~\cite{LTWZ13}, the method was refined to show that also 6-edge-connected graphs admit a nowhere-zero 3-flow. The question whether edge-connectivity 4 suffices remains open.
Another application of these modulo-$k$ orientations is the following decomposition result, which was shown by Thomassen in~\cite{Thb13} and also applied in~\cite{BMOW15+} and~\cite{Mer15+}.

\begin{theorem}\emph{\cite{Thb13}} \label{theorem:reduction1}
Let $G$ be a bipartite graph with partition classes $A_1$ and $A_2$, and size divisible by $m$. If $G$ is $(4\lambda + 6m)$-edge-connected, then $G$ can be decomposed into two $\lambda$-edge-connected graphs $G_1$ and $G_2$ such that $d(v,G_i)$ is divisible by $m$ for every $v$ in $A_i$.
\end{theorem}

By Theorems~\ref{theorem:bipartite} and~\ref{theorem:reduction1}, it is sufficient to prove Conjecture~\ref{conjecture:barat-thomassen} for bipartite graphs $G$ on vertex classes $A$ and $B$, where all vertices in $A$ have degree divisible by $m$, the size of $T$.

Let $T_A$ and $T_B$ denote the vertex classes of a bipartition of $T$. We may assume that $T_B$ contains a leaf.
The $T$-decompositions we are going to construct will respect the bipartitions of $G$ and $T$ in the
sense that the vertices corresponding to $T_A$ will lie in $A$ for each copy of $T$. We say that vertices $v\in V(G)$ and $t\in V(T)$ are \emph{compatible} if $v\in A$ and $t\in T_A$, or $v\in B$ and $t\in T_B$.

We study a specific kind of edge-colouring of $G$, that was introduced in~\cite{Mer15+}. Assuming $G$ is (improperly) edge-coloured, we denote by $d_i(v)$ the degree of vertex $v$ in colour $i$. 
For $t\in V(T)$, let $S(t)$ denote the set of edges incident with $t$. 

\begin{definition}
An edge-colouring $\phi : E(G) \rightarrow E(T)$ is called $T$-\textbf{equitable}, if for any compatible vertices $v\in V(G), t\in V(T)$ and $j,k\in S(t)$, we have $d_j(v)=d_k(v)$.
\end{definition}

It was shown by Merker~\cite{Mer15+}, using modulo-$k$ orientations, that highly edge-connected graphs admit $T$-equitable edge-colourings.

\begin{theorem}[Theorem 3.4 in~\cite{Mer15+}]\label{thm:merker}
For all natural numbers $m$ and $L$ there exists a natural number $f(m,L)$ such that the following holds:
If $G$ is an $f(m,L)$-edge-connected bipartite graph with bipartite classes $A$ and $B$ where all vertices in $A$ have degree divisible by $m$, and $T$ is a tree on $m$ edges with bipartite classes $T_A$ and $T_B$ where $T_B$ contains a leaf, then $G$ admits a $T$-equitable colouring where the minimum degree in each colour is at least $L$.
\end{theorem}

Notice that since we put no constraints on the degrees in $B$, it is necessary that the greatest common divisor of the degrees in $T_B$ is $1$ if we want to construct a $T$-equitable colouring. For this reason we chose the bipartition of $T$ so that $T_B$ contains a leaf.

If there exists a $T$-decomposition of a bipartite graph $G$ where all copies of $T$ are oriented the same way (with respect to the bipartite classes), then this gives rise to a $T$-equitable colouring of $G$. Vice versa, a $T$-equitable colouring can also be used to construct a $T$-decomposition. This was done in~\cite{Mer15+} for the case that the girth of $G$ is at least the diameter of $T$, and also in general for trees of diameter at most 4. 

In this paper we use probabilistic methods, inspired from those used in~\cite{BHLT15+}, to show that a $T$-equitable colouring can be turned into a $T$-decomposition whenever the minimum degree in each colour is large enough.

\begin{theorem}\label{theorem:main}
Let $T$ be a tree on $m$ edges and let $G$ be a bipartite graph admitting a $T$-equitable colouring. If the minimum degree in each colour is at least $10^{50m}$, then $G$ has a $T$-decomposition.
\end{theorem}

Combined with the previous theorems, Theorem \ref{theorem:main} completes the proof of the Bar\'at-Thomassen conjecture.

\begin{proof}[Proof of Theorem~\ref{theorem:result}:]
By Theorem~\ref{theorem:bipartite}, we may assume that $G$ is bipartite. We show that every $(4f(m,10^{50m})+6m)$-edge-connected bipartite graph has a $T$-decomposition, where $f$ is the function given by Theorem~\ref{thm:merker}. By Theorem~\ref{theorem:reduction1} we can decompose $G$ into two spanning $f(m,10^{50m})$-edge-connected graphs $G_1$ and $G_2$, such that in one side of the bipartition of each $G_i$ all vertices have degree divisible by $m$. By Theorem~\ref{thm:merker}, we can find a $T$-equitable colouring of $G_i$ in which the minimum degree in each colour is at least $10^{50m}$. This colouring can be turned into a $T$-decomposition by Theorem~\ref{theorem:main}.
\end{proof}


\section{Definitions and sketch of proof} \label{section:settings}

In our proof of Theorem~\ref{theorem:main}, a $T$-decomposition of a graph $G$ is obtained in two steps, which we describe more formally below.
In the first step we construct a decomposition of $G$ into so-called pseudo-copies of $T$, which are subgraphs that are in some sense homomorphic to $T$. Such a decomposition, which we call a $T$-pseudo-decomposition, can easily be obtained from a $T$-equitable colouring, see also~\cite{Mer15+}. Instead of choosing any such decomposition, we use probabilistic methods to find one in which the vast majority of pseudo-copies at every vertex are isomorphic to $T$. The details of this step can be found in Section 3.
In the second step, we use these isomorphic copies to repair the non-isomorphic copies of $T$ by making subgraph switches. While the switching itself is a deterministic operation, we again use probabilistic methods to find a suitable set of isomorphic copies. This part of the proof is detailed in Section 4.

\subsection*{Step 1: Finding a good $T$-pseudo-decomposition}

Let $G$ be a graph with a $T$-equitable colouring.
Recall that $S(t)$ denotes the set of edges incident with a vertex $t\in V(T)$. 
Let us denote by $N_i(v)$ the set of edges coloured $i$ incident with $v\in V(G)$. 
Furthermore, we set $N_{S(t)}(v) := \bigcup_{i\in S(t)} N_i(v)$ for every $t\in V(T)$ compatible with $v$.
Since the edge-colouring of $G$ is $T$-equitable, we have that $|N_{S(t)}(v)|$ is divisible by $|S(t)|$ for every compatible $v\in V(G)$, $t\in V(T)$. Thus, we partition $N_{S(t)}(v)$ into stars of size $|S(t)|$ that contain each of the colours in $S(t)$ exactly once. Let $\mathcal{S}$ be the collection of stars we get after having done this for every $v\in V(G)$ and compatible $t\in V(T)$. Consider an auxiliary graph $G_{\mathcal{S}}$ whose vertices are the stars in $\mathcal{S}$, and where two vertices are joined by an edge whenever the corresponding stars have an edge in common. By construction, each connected component of $G_{\mathcal{S}}$ is a tree isomorphic to $T$. For every connected component in $G_{\mathcal{S}}$, we take the union of all the stars corresponding to it in $G$. It is easy to see that this decomposes $G$ into parts of the same size as $T$. In fact even more is true: Each part is isomorphic to a graph obtained from $T$ by identifying vertices. This motivates the following definition.

\begin{definition}
A graph $H$ is a \textbf{pseudo-copy} of $T$, if there exists a surjective graph homomorphism $h:V(T)\rightarrow V(H)$ that induces a bijection between $E(T)$ and $E(H)$.
\end{definition}

In other words, a graph $H$ is a pseudo-copy of $T$, if it is isomorphic to a multigraph obtained from $T$ by identifying vertices and keeping all edges. 
We also refer to pseudo-copies of $T$ as \textit{pseudo-trees}. 
The term \textit{$T$-pseudo-decomposition} denotes a decomposition where each part is a pseudo-copy of $T$. 
Given a $T$-equitable colouring of $G$, the construction above results in a $T$-pseudo-decomposition of $G$. 

Notice that it might be the case that a graph $H$ can be considered as a pseudo-copy of $T$ in different ways if there exists more than one homomorphism from $T$ to $H$ with the required properties. However, we will only consider homomorphisms that induce the same edge-colouring of $H$ as the given $T$-equitable colouring. Furthermore, we only consider pseudo-copies of $T$ in $G$ that respect the bipartition in the sense that vertices corresponding to $T_A$ always lie in $A$.

Let $\mathcal{P}$ be a $T$-pseudo-decomposition of $G$. For every compatible $v\in V(G)$ and $t\in V(T)$, we denote by $N_\mathcal{P}(v|t)$ the set of pseudo-trees in $\mathcal{P}$ in which $v$ is the image of $t$.
Let $d_\mathcal{P}(v|t)= | N_\mathcal{P}(v|t)|$. 
Clearly, for any two different vertices $u$ and $v$ of $G$, we have $N_\mathcal{P}(u|t)\cap N_\mathcal{P}(v|t)=\emptyset$. Notice also that 
$$\bigcup_{v\in G} N_\mathcal{P}(v|t)=\mathcal{P}$$
for every $t\in V(T)$.

So far we have explained how $T$-pseudo-decompositions can be obtained from a $T$-equitable colouring. We denote such a resulting $T$-pseudo-decomposition $\mathcal{P}$ of $G$ by $\mathcal{H}\cup \mathcal{I}$, where $\mathcal{I}$ denotes the collection of pseudo-copies that are isomorphic to $T$ and $\mathcal{H}$ denotes the collection of the remaining pseudo-copies.

If the minimum degree in each colour is large in the $T$-equitable colouring, then there are many possibilities at every vertex to decompose its incident edges into stars. 
Using probabilistic methods, we find a $T$-pseudo-decomposition where $d_\mathcal{H}(v|t)\le \varepsilon d_{\mathcal{I}}(v|t)$ for some given $\varepsilon >0$ and every compatible $v\in V(G)$, $t\in V(T)$. 
Now for every non-isomorphic copy $H\in N_\mathcal{H}(v|t)$, there are many copies isomorphic to $T$ in $N_\mathcal{I}(v|t)$. 
We will use one of these isomorphic copies to improve the $T$-pseudo-decomposition by repairing $H$.
This is done by a subgraph switch operation which is explained in more detail in Step 2.
However, if the trees in $N_\mathcal{I}(v|t)$ overlap too much, then we might not be able to make any switch that improves the $T$-pseudo-decomposition. 
To avoid this, we need to find a large set of isomorphic copies in $N_\mathcal{I}(v|t)$ that pairwise intersect only in $v$. 
To measure how much the pseudo-trees in a $T$-pseudo-decomposition overlap, we use the following concept that was introduced in~\cite{BHLT15+}.

\begin{definition}
Let $\mathcal{P}$ be a collection of pseudo-copies of $T$ in $G$, and $v\in V(G)$ and $t\in V(T)$ be compatible vertices. The \textbf{conflict ratio} of $v$ with respect to $t$, denoted by $\mbox {\rm conf}_{\mathcal{P}}(v|t)$, is defined by
$$\mbox {\rm conf}_{\mathcal{P}}(v|t):=\frac{\max_{u\ne v} \big| \{T\in { N} _\mathcal{P}(v|t) : u \in V(T)\}\big|}{d_\mathcal{P}(v|t)}.$$
\end{definition}

Intuitively, ${\rm conf}_{\mathcal{P}}(v|t)$ measures the maximum proportion of pseudo-copies in $N_\mathcal{P}(v|t)$ in
which some fixed vertex $u$ appears. Clearly, we always have $0\leq {\rm conf}_{\mathcal{P}}(v|t) \leq 1$. If $v$ and $t$ are not compatible, then we set ${\rm conf}_{\mathcal{P}}(v|t)=0$. Globally, we define ${\rm conf}(\mathcal{P}|t):=\max_{v}\mbox {conf}_{\mathcal{P}}(v|t)$ and ${\rm conf}(\mathcal{P}):=\max_{t}\mbox {conf}(\mathcal{P}|t)$.

To ensure that the isomorphic copies in the $T$-pseudo-decomposition $\mathcal{H}\cup \mathcal{I}$ are sufficiently spread out, we also require ${\rm conf}(\mathcal{I}) \leq \delta$ for some given $\delta >0$. In Section 3, we prove that such a $T$-pseudo-decomposition can always be obtained provided the minimum degree in each colour is large enough.

\begin{lemma}\label{dense}
Let $T$ be a tree on $m$ edges and $\varepsilon$, $\delta$ real numbers with $0 < \varepsilon, \delta < 1$.
Let $G$ be a $T$-equitably coloured bipartite graph where the minimum degree in each colour is at least $(10m)^{18} (\varepsilon\delta)^{-6}$.
Then $G$ admits a $T$-pseudo-decomposition $\mathcal{H} \cup \mathcal{I}$, where $\mathcal{I}$ denotes the collection of isomorphic copies of $T$, such that:
\begin{enumerate}
\item[(1)] for every compatible $v\in V(G)$ and $t\in V(T)$, we have $d_\mathcal{H}(v|t)\le \varepsilon d_{\mathcal{I}}(v|t)$;
\item[(2)] ${\rm conf}(\mathcal{I}) \leq \delta$.
\end{enumerate} 
\end{lemma}

\subsection*{Step 2: Repairing non-isomorphic copies}

For this part of the proof we label the vertices $t_0,\ldots ,t_m$ of $T$ so that, for every $i\in \{1,\ldots m\}$, the subgraph induced by $t_0,\ldots ,t_i$ is connected. Such an ordering can for example be obtained by applying a breadth-first search algorithm from some vertex $t_0$ of $T$.
We also label the edges of $T$ so that $e_i$ denotes the edge joining $t_i$ with $T[t_0,\ldots ,t_{i-1}]$ for every $i\in \{1,\ldots m\}$. 
To indicate at which place a pseudo-copy $H$ fails to be isomorphic to $T$, we introduce the following definitions.

\begin{definition}
Let $H$ be a pseudo-copy of $T$, and let $v_i$ denote the image of $t_i$ in $H$ for every $i\in \{0,\ldots ,m\}$. For $i\in \{1,\ldots ,m\}$, we say that $H$ is \textbf{$i$-good} if the vertices $v_0,\ldots ,v_i$ are pairwise distinct. If $H$ is not $i$-good, then we say that $H$ is \textbf{$i$-bad}.
\end{definition}
Note that since $G$ does not have multiple edges, every pseudo-copy of $T$ in $G$ is 2-good. Moreover, since $G$ is bipartite, every pseudo-copy of $T$ in $G$ is even 3-good. 

The idea is to use isomorphic copies to repair the pseudo-trees that are not isomorphic to $T$. We start by considering all pseudo-trees in $\mathcal{H}$ that are 4-bad. For each such $H$, we will find an isomorphic copy $f(H)$ in $\mathcal{I}$ such that $H\cup f(H)$ can be written as the union of two 5-good pseudo-copies of $T$, say $H_1\cup H_2$. We then remove $H$ from $\mathcal{H}$ and $f(H)$ from $\mathcal{I}$, and add $\{H_1,H_2\}$ to $\mathcal{H}$. 
The technical definition of this so-called \textit{switch} is given below.
We use this operation for all 4-bad pseudo-copies of $T$ in $\mathcal{H}$. Let $\mathcal{H}'\cup \mathcal{I}'$ denote the resulting $T$-pseudo-decomposition, where $\mathcal{I}'$ again contains only isomorphic copies of $T$ and all pseudo-copies in $\mathcal{H}'$ are 4-good. 
We repeat this step, this time repairing all 5-bad pseudo-copies in $\mathcal{H}'$ by using isomorphic copies in $\mathcal{I}'$. We continue like this until we get a $T$-pseudo-decomposition in which all pseudo-copies are $m$-good and thus isomorphic to $T$.

To make sure that we can perform a switch between $H$ and $f(H)$, we need $f(H)$ to satisfy certain properties. Let $v_j$ denote the image of $t_j$ in $H$ for $j\in\{0,\ldots ,m\}$ and suppose $i$ is chosen minimal such that $H$ is $i$-bad. By the choice of our labelling, there exists $i'\in\{0,\ldots, i-1\}$ with $t_{i'}t_i \in E(T)$. To ensure that $v_i$ is distinct from the previous vertices $v_0,\ldots ,v_{i-1}$, we want to choose a different edge corresponding to $e_i$ at $v_{i'}$. Since we take this edge from $f(H)$, we want $f(H)$ to also use the vertex $v_{i'}$ as image of $t_i'$. However, this should be the only point of intersection with $H$ to ensure that both copies will be $i$-good after the switch.

More precisely, for every edge $e_i\in E(T)$, let $T^{i-}$ denote the connected component of $T-e_i$ containing $t_0$. Let $T^{i+}$ be the subgraph of $T$ induced by $E(T)\setminus E(T^{i-})$. If $H$ is a pseudo-copy of $T$, then we denote the images of $T^{i-}$ and $T^{i+}$ under the homomorphism by $H^{i-}$ and $H^{i+}$. Now we are ready to define the switching operation.

\begin{definition}
Let $\mathcal{H}$ be a collection of pseudo-copies of $T$ in $G$ and $i\in\{1,\ldots ,m\}$. Let $t_{i'}$ be the endpoint of the edge $e_i$ that is different from $t_i$. Suppose $H_1,H_2\in N_{\mathcal{H}}(v|t_{i'})$ for some $v\in V(G)$. The \textbf{$i$-switch} of $\{H_1,H_2\}$ is defined by
$${\rm sw}_i(\{H_1,H_2\})=\{H_1^{i+}\cup H_2^{i-},H_1^{i-}\cup H_2^{i+}\}.$$
\end{definition} 

By making an $i$-switch between two pseudo-copies $H$ and $f(H)$, their vertices corresponding to $v_0,\ldots ,v_{i-1}$ remain unchanged. In particular, if both $H$ and $f(H)$ are $(i-1)$-good, then also both copies in ${\rm sw}_i(\{H,f(H)\})$ will be $(i-1)$-good. Moreover, if $H\cap f(H) = \{v_{i'}\}$, then after the switch both pseudo-trees will be $i$-good. Notice that neither of the two new pseudo-trees is necessarily still isomorphic to $T$. In particular, the collection of isomorphic copies might shrink with every step of the repairing process. 

If the pseudo-trees in $\mathcal{I}$ overlap too much, we might not be able to find a single pseudo-tree $f(H)$ in $\mathcal{I}$ with $H\cap f(H) = \{v_{i'}\}$. A sufficiently low conflict ratio of $\mathcal{I}$ ensures that we can find such a function $f:\mathcal{H}\rightarrow \mathcal{I}$. However, to continue this process we also need that the remaining collection of isomorphic copies $\mathcal{I}\setminus f(\mathcal{H})$ has a low conflict ratio. To this end we use the Local Lemma to prove the following lemma in Section 4.

\begin{lemma}\label{lemma-switchset}
Let $T$ be a tree on $m$ edges and $\varepsilon$, $\delta$ positive real numbers with $\varepsilon +\delta m<\frac{1}{2}$ and $\varepsilon m < 1$. Let $\mathcal{H}$ and $\mathcal{H}'$ be collections of pseudo-copies of $T$ in $G$ with ${\rm conf}(\mathcal{H}') \leq \delta$ and $d_{\mathcal{H}'}(v|t)> {\rm max}\{22/\varepsilon^{7},  d_{\mathcal{H}}(v|t)/\varepsilon\}$ for each compatible $v\in V(G)$, $t\in V(T)$. \\
For every $t\in V(T)$, there exists an injective function $f_t:\mathcal{H}\rightarrow \mathcal{H}'$
such that 
\begin{itemize}
\item $f_t(N_{\mathcal{H}}(v|t)) \subset N_{\mathcal{H}'}(v|t)$ for every $v\in V(G)$ compatible with $t$,
\item $H \cap f_{t}(H) = \{v\}$ for every $H\in N_{\mathcal{H}}(v|t)$, and
\item $d_{f_t(\mathcal{H})}(v|t') \leq 3\varepsilon d_{\mathcal{H}'}(v|t')$ for every compatible $v\in V(G)$, $t'\in V(T)$.
\end{itemize}
\end{lemma}

By using Lemma~\ref{lemma-switchset} with $\mathcal{H}'=\mathcal{I}$, we find a collection $f(\mathcal{H})$ in which the degrees are low compared to the degrees in $\mathcal{I}$. Thus, the conflict ratio of the collection of isomorphic copies only increases by a constant factor after each step of the repairing process. By choosing $\varepsilon$ and $\delta$ sufficiently small, the proof of Theorem~\ref{theorem:main} will follow from Lemma~\ref{dense} and repeated applications of Lemma~\ref{lemma-switchset}. The details can be found at the end of Section 4.


\section{Finding a good $T$-pseudo-decomposition}

Given a graph with a $T$-equitable colouring and large minimum degree in each colour, we construct a $T$-pseudo-decomposition satisfying the conditions in Lemma~\ref{dense}. 
As described in Step 1 of Section 2, every $T$-equitable colouring gives rise to several $T$-pseudo-decompositions. 
We form the pseudo-copies of $T$ by grouping the edges at every vertex randomly into rainbow stars. 
If the degrees in each colour are large enough, we can ensure that most of the resulting pseudo-trees are isomorphic to $T$ and also the conflict ratio of the resulting $T$-pseudo-decomposition is small. 
The proof of this is essentially an application of the Local Lemma.

\begin{proposition}[Symmetric Local Lemma] \label{prop: SymLovasz}
Let $A_1,..., A_n$ be events in some probability space $\Omega$
with $\PP[A_i] \leq p$ for all $i\in \{1,\ldots ,n\}$.
Suppose that each $A_i$ is mutually independent of all but at most $d$ other events $A_j$. If $4pd < 1$, then $\PP[\cap_{i=1}^{n} \overline{A_i}] > 0$.   
\end{proposition}

The bad events in this case are of the form that many copies in $N_{\mathcal{H}}(v|t)$ are either not isomorphic to $T$ or have a vertex different from $v$ in common.
To show that each event occurs with low probability, we make use of an inequality due to McDiarmid \cite{M02} (see also~\cite{MR02}). 

\begin{proposition}[McDiarmid's Inequality (simplified version)]
Let $X$ be a non-negative random variable, not identically 0, which is determined by 
$m$ independent permutations $\Pi_1,..., \Pi_m$.  A \emph{choice} is  the position that a particular element gets mapped to in a permutation. If there exist $d, r >0$ such that 
\begin{itemize}
 \item interchanging two elements in any one permutation can affect $X$ by at most $d$, and
 \item for any $s>0$, if $X \geq s$ then there is a set of at most $rs$ choices whose outcomes certify that $X\geq s$,
\end{itemize}
then for any $0 \leq \lambda \leq \EE[X]$,
$$ \PP\left[|X-\EE[X]| > \lambda + 60d \sqrt{r\EE[X]}\right] \leq 4 e^{-\tfrac{\lambda^2}{8d^2r\EE[X]}}\,.$$
\end{proposition}

A necessary condition to apply the Local Lemma is that each event is mutually independent of most other events. To make sure that this is the case, we start the proof by partitioning the edges at each vertex into so-called \textit{fans} of roughly the same size. 
Recall that for $v\in V(G)$ and $t\in V(T)$, we denote by $N_i(v)$ the edges coloured $i$ incident with $v$ in $G$, and by $S(t)$ the set of edges incident with $t$ in $T$.

\begin{proof}[Proof of Lemma~\ref{dense}:]
Set $c=\lceil (10m)^9(\varepsilon\delta)^{-3}\rceil$. 
For every $v\in V(G)$ and colour $i$, we choose $r_{v,i}\in \{0,\ldots ,c-2\}$ such that $d_i(v)\equiv r_{v,i}$ (mod $c-1$). 
Since the minimum degree in each colour in $G$ is greater than $c^2$, we can partition every set $N_i(v)$ into subsets of size $c$ and $c-1$ so that precisely $r_{v,i}$ of them have size $c$. We call these subsets {\it $i$-blades}. 
Note that an edge $uv$ of colour $i$ in $G$ appears both in an $i$-blade of $N_i(u)$ as well as in an $i$-blade of $N_i(v)$, but we do not require these two $i$-blades to have the same size.

For every compatible $t\in V(T)$, $v\in V(G)$, and $i,j\in S(t)$, we have $d_i(v)=d_j(v)$ since the colouring is $T$-equitable. Thus, the number of $i$-blades of size $c$ (respectively, of size $c-1$) in the partition of $N_i(v)$ is equal to the number of $j$-blades of size $c$ (respectively, of size $c-1$) in the partition of $N_j(v)$. 
We can therefore partition the edges of $N_{S(t)}(v)$ into {\it fans}, which are unions of blades of the same size, such that every fan contains precisely one $i$-blade for every $i \in S(t)$. 
In other words, a fan $\varphi$ at a vertex $v$ (with relation to $t$) is a subset of $N_{S(t)}(v)$ of size $c|S(t)|$ or $(c-1)|S(t)|$ such that all colours in $S(t)$ appear $c$ times or $c-1$ times in $\varphi$. We also call $\varphi$ a \textit{$t$-fan} to indicate the colours appearing in $\varphi$. 

For every compatible $t\in V(T)$, $v\in V(G)$, and every $t$-fan $\varphi$ at $v$, we uniformly at random partition the edges in $\varphi$ into rainbow stars of size $|S(t)|$. More precisely, for every $i\in S(t)$ 
we choose a permutation $\Pi_{\varphi ,i}$ independently and uniformly at random from all permutations on $c$ elements (respectively, on $c-1$ elements if the blades of $\varphi$ have size $c-1$). 
By labelling the edges of each blade, each permutation $\Pi_{\varphi ,i}$ corresponds to an ordering of the edges of the $i$-blade of $\varphi$. 
Now we partition the edges of $\varphi$ into stars of size $|S(t)|$ by grouping the edges of different blades that were mapped to the same position. 
In other words, for every $s\in \{1,\ldots ,c\}$ (respectively, $s\in \{1,\ldots ,c-1\}$) we form a star by choosing for every $i\in S(t)$ the edge labelled $\Pi_{\varphi ,i}(s)$ in the $i$-blade of $\varphi$. 
These stars are centered at $v$ and each colour in $S(t)$ appears precisely once.
Note that every edge $uv\in E(G)$ belongs to exactly two stars, one centered at $u$ and one centered at $v$. 
As described in Step~1 in Section~\ref{section:settings}, these stars correspond to a $T$-pseudo-decomposition of $G$ in a canonical way. 
All that remains to show is that there exists an outcome of the random permutations such that the resulting $T$-pseudo-decomposition 
satisfy the conditions (1) and (2) of Lemma~\ref{dense}.

We denote the set of pseudo-trees using edges of a fan ${\varphi}$ by $\mathcal{T}_{\varphi}$. 
Note that $|\mathcal{T}_{\varphi}|$ is either equal to $c-1$ or $c$. Now we formally define what the bad events at a $t$-fan $\varphi$ at a vertex~$v$ are.
Let $A_{\varphi}$ be the event that more than $2m^2 c^{2/3}$ of the pseudo-copies in $\mathcal{T}_{\varphi}$ are not isomorphic to $T$.
Let $B_{\varphi}$ be the event that there exists a vertex $u\in V(G)$ with $u\neq v$ such that more than $2m c^{2/3}$ pseudo-copies in $\mathcal{T}_{\varphi}$ contain $u$. 
Finally, let $C_{\varphi} = A_{\varphi}\cup B_{\varphi}$. We will prove the following two statements.

\begin{description}
\item[Claim 1:] Each $C_{\varphi}$ is mutually independent of all but at most $4(cm)^{2m}$ other events $C_{\psi}$.
\item[Claim 2:] $\PP[C_{\varphi}] < 9(cm)^m m e^{-c^{2/3}/32}$.
\end{description}

\noindent
Before we proceed to prove these claims, let us note that they allow us to use the Local Lemma to get our desired $T$-pseudo-decomposition $\mathcal{H}\cup \mathcal{I}$.
Indeed, since $e^{-x}<\frac{(9m)!}{x^{9m}}$ for $x>0$, we have
\begin{eqnarray*}
4\cdot 4(cm)^{2m}\cdot \PP[C_{\varphi}] & < & 2^8 \cdot (cm)^{3m} \cdot m\cdot e^{-c^{2/3}/32}\\
	& < & 2^{45m+8} \cdot \left(\frac{m}{c}\right)^{3m} \cdot m \cdot (9m)! \\
	& < & \left( 2^{18} \cdot \frac{m}{c} \cdot (9m)^3\right)^{3m} \\
	& < & \left(\frac{10^9 m^4}{c}\right)^{3m}\\
	& < & 1\,,
\end{eqnarray*}
where the last inequality follows from $c \geq (10m)^9$.
Thus, the symmetric version of the Local Lemma yields a $T$-pseudo-decomposition $\mathcal{H}\cup \mathcal{I}$ for which none of the events $C_{\varphi}$ holds. Now $\mathcal{H}\cup \mathcal{I}$ has the desired properties:
\begin{itemize}
\item Since $A_{\varphi}$ does not hold for any $\varphi$, at most $2m^2c^{2/3}$ of the pseudo-copies in $\mathcal{T}_{\varphi}$ are not isomorphic to $T$. Since $c\geq (10m)^9\varepsilon^{-3}$, we have $2m^2c^{2/3} < \frac{\varepsilon}{1+\varepsilon}c$. Thus, less than $\frac{\varepsilon}{1+\varepsilon}c$ of the pseudo-copies in $\mathcal{T}_{\varphi}$ are in $\mathcal{H}$, while at least $\frac{1}{1+\varepsilon}c$ of them are in $\mathcal{I}$. This holds for every $t$-fan at $v$, so we have $d_{\mathcal{H}}(v|t) < \varepsilon d_{\mathcal{I}}(v|t)$.
\item Since $B_{\varphi}$ does not hold for any $\varphi$, there are at most $2m c^{2/3}$ trees in $\mathcal{T}_{\varphi}$ containing a given vertex $u$ different from $v$. As argued above, at least $\frac{c}{1+\varepsilon}$ of the pseudo-copies in $\mathcal{T}_{\varphi}$ are in $\mathcal{I}$. Since $c\geq (10m)^9(\varepsilon\delta)^{-3}$, we have $2m c^{2/3} < \delta \frac{c}{1+\varepsilon}$. Thus, the proportion of trees in $\mathcal{T}_\varphi \cap \mathcal{I}$ containing $u$ is less than $\delta$. This is true for every $t$-fan at $v$, so we have $$\frac{\big| \{H\in { N} _\mathcal{I}(v|t) : u \in V(H)\}\big|}{d_{\mathcal{I}}(v|t)}\le \delta$$ and thus  ${\rm conf}(\mathcal{I}) \le \delta$.
\end{itemize}

It remains to verify Claims 1 and 2.
We begin by proving Claim~1.

\begin{proof}[Proof of Claim 1:]
The structure of $\mathcal{T}_{\varphi}$ depends on permutations in different fans. 
Let $J(\varphi)$ denote the set of fans $\psi$ for which there exists an outcome of the random permutations 
such that $\mathcal{T}_{\varphi}\cap \mathcal{T}_{\psi}$ is non-empty. 
Since each fan consists of at most $cm$ edges, there are at most $cm+(cm)^2+\ldots +(cm)^m$ fans we can reach from $\varphi$ via a path of length at most $m$. 
Thus, 
$$|J(\varphi)|\leq cm+(cm)^2+\ldots +(cm)^m < 2(cm)^m\,.$$ 
This shows that there are at most $2(cm)^m$ fans where the outcome of the permutation affects the structure of $\mathcal{T}_{\varphi}$.
The same calculation shows that each permutation affects the structure of at most $2(cm)^m$ sets $\mathcal{T}_{\psi}$. 
Hence, the event $C_{\varphi}$ is mutually independent of all but at most $4(cm)^{2m}$ other events $C_{\psi}$.
\end{proof}

Before we prove Claim 2, let us introduce more terminology.
Let $t_i$ and $t_j$ be two distinct vertices of $T$. Notice that $t_i$ or $t_j$ could be equal to $t$.
We say that a pseudo-copy $H$ of $T$ is \emph{$(t_i,t_j)$-bad} if the images of $t_i$ and $t_j$ in $H$ are identical.
For a $t$-fan $\varphi$ at a vertex $v$, let $A_{\varphi}(t_i,t_j)$ be the event that the number of $(t_i,t_j)$-bad pseudo-trees in $\mathcal{T}_{\varphi}$ is greater than $2c^{2/3}$. 
For a vertex $u\in V(G)$ with $u\neq v$, let $B_{\varphi}(u|t_i)$ be the event that the number of pseudo-trees in $\mathcal{T}_{\varphi}$ in which $u$ is the image of $t_i$ is greater than $2c^{2/3}$. 
The proof of Claim 2 consists of two parts:

\begin{description}
\item[Claim 2A:] $\PP[A_{\varphi}(t_i,t_j)] < 4e^{-c^{2/3}/32}$ for every $t_i,t_j\in V(T)$ with $t_i\neq t_j$. 
\item[Claim 2B:] $\PP[B_{\varphi}(u|t_i)] < 4e^{-c^{2/3}/8}$ for every $u\in V(G)$, $t_i\in V(T)$ and $u\neq v$.
\end{description}

The proofs of Claims 2A and 2B use McDiarmid's inequality and have a very similar structure. We will therefore present all the details in the proof of Claim 2A, and only point out the differences in the proof of Claim 2B.

\begin{proof}[Proof of Claim 2A:] 
Fix $t_i$ and $t_j$ as different vertices of $T$. 
Let $P_i$ and $P_j$ denote the paths in $T$ from $t$ to $t_i$ and $t_j$. In the case that one is a subpath of the other, we may assume that $P_i$ is contained in $P_j$. 
Let $t_{j'}$ denote the second last vertex of $P_j$ and let $j$ denote the edge joining $t_{j'}$ and $t_j$. 
Now $T-j$ consists of two components, one of which contains $t_j$ while the other one contains $t$, $t_i$, and $t_{j'}$. 

Let $\pi$ be a fixed outcome of all permutations apart from those at the $j$-blades of $t_{j'}$-fans. 
In other words, given $\pi$, we only need to know the outcome of the permutations $\Pi_{\psi ,j}$ for every $t_{j'}$-fan $\psi$ to construct the $T$-pseudo-decomposition.
For any such outcome $\pi$, we will show that the conditional probability $\PP[A_{\varphi}(t_i, t_j)|\pi]$ is at most $4e^{-c^{2/3}/32}$. 
Clearly, since we condition on an arbitrary but fixed event, this uniform bound implies Claim 2A.

Let $T'$ denote the component of $T-j$ containing $t$, $t_i$ and $t_{j'}$, and let $T''$ denote the subgraph of $T$ induced by $E(T)\setminus E(T')$. 
Let $\mathcal{T}'_{\varphi}$ denote the images of $T'$ in the pseudo-trees of $\mathcal{T}_{\varphi}$. 
By fixing $\pi$, the set $\mathcal{T}'_{\varphi}$ is also fixed.
The permutations of the $j$-blades at the $t_{j'}$-fans only decide how the images of $T'$ and $T''$ get matched at the $t_{j'}$-fans.

Let $\Psi$ denote the set of $t_{j'}$-fans which contain edges of pseudo-copies in $\mathcal{T}_\varphi$. 
Note that also the set $\Psi$ is completely determined by $\pi$. 
Let $X_{\varphi}$ denote the random variable counting the number of $(t_i,t_j)$-bad pseudo-trees in $\mathcal{T}_{\varphi}$ conditional on $\pi$. 
Notice that $X_{\varphi}$ only depends on the random permutations $\Pi_{\psi ,j}$ with $\psi \in \Psi$. 

For each pseudo-tree $H\in \mathcal{T}_{\varphi}'$ at a $t_{j'}$-fan $\psi \in \Psi$, we already know what the image of $t_i$ in $H$ is.
There are $c-1$ or $c$ different images of $T''$ that could get matched to $H$ at $\psi$, each having a distinct vertex as image of $t_j$.
Thus, there are at least $c-1$ different vertices that could be the image of $t_j$ in $H$. 
Since the permutation $\Pi_{\psi ,j}$ is chosen uniformly at random, the probability that $H$ will be part of a $(t_i,t_j)$-bad pseudo-tree is at most $\frac{1}{c-1}$.
Now, by linearity of expectation,
$$\mathbb{E}[X_{\varphi}] \leq |\mathcal{T}_{\varphi}| \cdot \frac{1}{c-1} \leq \frac{c}{c-1}.$$
We will apply McDiarmid's inequality to the random variable $Y_{\varphi}$ defined by $ Y_{\varphi}:= X_{\varphi} + c^{2/3}$. Clearly $\mathbb{E}[Y_{\varphi}] = \mathbb{E}[X_{\varphi}] + c^{2/3}$.
Only the permutations $\Pi_{\psi ,j}$ with $\psi \in \Psi$ affect $X_{\varphi}$ and thus $Y_{\varphi}$. If two elements in one of these permutations are interchanged, then the structure of two pseudo-trees in $\mathcal{T}_{\varphi}$ changes. In particular, the number of $(t_i,t_j)$-bad trees in $\mathcal{T}_{\varphi}$ changes by at most 2. Thus, we can choose $d=2$ in McDiarmid's inequality.

If $Y_{\varphi} \geq s$, then $X_{\varphi} \geq s - c^{2/3}$, and thus at least $s - c^{2/3}$ of the pseudo-trees in $\mathcal{T}_{\varphi}$ are $(t_i, t_j)$-bad. 
Let $H'\in \mathcal{T}_{\varphi}'$ be a part of a pseudo-tree $H$ that is counted by $X_{\varphi}$. Let $v_i$ and $v_j$ denote the images of $t_i$ and $t_j$ in $H$. 
To verify that $H$ is $(t_i, t_j)$-bad, we only need to know which edge in the $j$-blade of $\psi$ gets mapped to the same position as the edges in $H'$ in other blades of $\psi$. 
In other words, the vertex $v_j$ is determined by the position of one element in the permutation $\Pi_{\psi ,j}$, and thus $v_i=v_j$ can be certified by a single outcome.
Thus, $X_{\varphi} \geq s - c^{2/3}$ can be certified by the outcomes of $s - c^{2/3} < s$ choices and we can choose $r=1$ in McDiarmid's inequality.

By applying McDiarmid's inequality to $Y_{\varphi}$ with $\lambda =\EE[Y_{\varphi}]$, $d=2$, $r=1$, we get 
$$\PP\left[|Y_{\varphi}- \EE[Y_{\varphi}]| > \EE[Y_{\varphi}] + 120 \sqrt{\EE[Y_{\varphi}]}~\right] \leq 4 e^{-\tfrac{\EE[Y_{\varphi}]}{32}} \leq 4e^{-\tfrac{c^{2/3}}{32}}\,.$$
Since $c \geq 10^9$ and $\EE[Y_{\varphi}] \leq c^{2/3}+\frac{c}{c-1}$, we have $120\sqrt{\EE[Y_{\varphi}]} < \frac{1}{2}\EE[Y_{\varphi}] $ which implies 
$$\PP\left[X_{\varphi} > 2c^{2/3} \right] = \PP\left[Y_{\varphi} > 3c^{2/3}\right]  < \PP\left[ |Y_{\varphi} - \EE[Y_{\varphi}]| > \frac{3}{2}\EE[Y_{\varphi}] \right]  \leq 4e^{-\tfrac{c^{2/3}}{32}}\,.$$
Now $\PP[A_{\varphi}(t_i,t_j)|\pi] < 4e^{-c^{2/3}/32}$ and Claim 2A follows. 
\end{proof}


\begin{proof}[Proof of Claim 2B:]
Let $t_i\in V(T)$ be a fixed vertex different from $t$. Let $P$ denote the path from $t$ to $t_i$ in $T$.
Let $t_j$ denote the second last vertex of $P$ and let $i$ denote the edge joining $t_j$ and $t_i$.
Now $T-i$ consists of two components, one of which contains $t$ and $t_j$ while the other one contains $t_i$.
Let $\pi$ be any fixed outcome of all permutations apart from those at the $i$-blades of $t_{j}$-fans. 
We show that the conditional probability $\PP[B_{\varphi}(u|t_i)|\pi ]$ is at most $4e^{-c^{2/3}/8}$. 
As $\pi$ is arbitrary, this implies the general bound $\PP[B_{\varphi}(u|t_i)] < 4e^{-c^{2/3}/8}$. 

Let $X_{\varphi}$ denote the random variable conditional on $\pi$ which counts the number of pseudo-trees in $\mathcal{T}_{\varphi}$ where $u$ is the image of $t_i$. 
The vertex $u$ appears at most once in each $t_j$-fan, so by linearity of expectation we have 
$$\mathbb{E}[X_{\varphi}] \leq |\mathcal{T}_{\varphi}| \cdot \frac{1}{c-1} \leq \frac{c}{c-1}\,.$$

We apply McDiarmid's inequality to the random variable $X_{\varphi} + c^{2/3}$. Swapping two positions in a permutation $\Pi_{\psi ,i}$ can affect $X_{\varphi}$ by at most~$1$ since $u$ is incident to at most one edge of the $i$-blade of $\psi$.
If $X_{\varphi} + c^{2/3} \geq s$, then this can be certified by revealing at most $s$ positions in the random permutations. Thus, applying McDiarmid's inequality to the random variable $X_{\varphi} + c^{2/3}$ with $\lambda =\EE[X_{\varphi}]+c^{2/3}$, $r=1$, $d=1$ yields 
$$\PP\left[X_{\varphi} > 2c^{2/3} \right] \leq 4e^{-c^{2/3}/8}\,.$$
Now $\PP[B_{\varphi}(u|t_i)|\pi] < 4e^{-c^{2/3}/8}$ and Claim 2B follows.
\end{proof}

Now the proof of Claim 2 follows easily from Claims 2A and 2B.

\begin{proof}[Proof of Claim 2:]
By Claim 2A, we have 
$$\PP[A_{\varphi}]\le \PP\left[\bigcup_{\forall i < j}A_{\varphi}(t_i,t_j)\right]  \le \sum_{\forall i < j} \PP\left[A_{\varphi}(t_i,t_j)\right] < 4m^2e^{-c^{2/3}/32}\,.$$
Let $B_{\varphi}(u)$ be the event that the number of pseudo-trees in $\mathcal{T}_{\varphi}$ containing $u$ is greater than $2mc^{2/3}$. Since $u$ cannot be the image of $t$, we have, by Claim 2B,
$$\PP[B_{\varphi}(u)] \leq \PP\left[\bigcup_{\forall i}B_\varphi(u|t_i)\right]\le \sum_{\forall i}\PP[B_\varphi(u|t_i)]< 4m e^{-c^{2/3}/8}\,.$$
Since each fan consists of at most $cm$ edges, there are at most $cm + (cm)^2+\ldots + (cm)^m$ vertices we can reach from $\varphi$ via a path of length at most $m$. Thus, there are less than $2(cm)^m$ vertices $u$ for which $\PP[B_{\varphi}(u)]$ could be positive.
In particular, we have 
$$\PP[B_\varphi]=\PP\left[\bigcup_{\forall u, u\neq v}B_\varphi(u)\right]\le \sum_{\forall u, u\neq v}\PP[B_\varphi(u)]< 8(cm)^m m e^{-c^{2/3}/8}$$
and Claim 2 follows from $\PP[C_\varphi] \leq \PP[A_\varphi] + \PP[B_\varphi]$. 
\end{proof}

This concludes the proof of Lemma~\ref{dense}.
\end{proof}


\section{Repairing non-isomorphic copies}

Let $\mathcal{H}\cup \mathcal{I}$ be the $T$-pseudo-decomposition given by Lemma~\ref{dense}. 
As described in Step 2 in Section 2, we use copies in $\mathcal{I}$ to repair the pseudo-trees in $\mathcal{H}$ that are not isomorphic to $T$. 
We apply Lemma~\ref{lemma-switchset} to show the existence of a suitable subset of $\mathcal{I}$ to perform the switches.
The proof of Lemma~\ref{lemma-switchset} relies on the following probabilistic tools, see also~\cite{MR02}.

\begin{proposition}[Lov\'{a}sz Local Lemma] \label{prop: Lovasz}
Let $A_1,..., A_n$ be a finite set of events in some probability space $\Omega$,
and suppose that for some $J_i \subset [n]$, $A_i$ is mutually independent of 
$\{A_j : j \notin J_i \cup \{i\}\}$. If there exist real numbers $x_1,..., x_n$ in $(0,1)$ such that 
$\PP[A_i] \leq x_i \prod_{j \in J_i} (1-x_j)$ for every $i\in\{1,...,n\}$, then $\PP[\cap_{i=1}^{n} \overline{A_i}] > 0$. 
\end{proposition}

\begin{proposition}[Simple Concentration Bound] Let $X$ be a random variable determined by $n$ independent
trials $T_1,..., T_n$ such that changing the outcome of any one trial $T_i$
can affect $X$ by at most $c$. 
\noindent Then
$$ \PP[|X-\mathbb{E}[X]| > \lambda] \leq 2e^{-\lambda^2 / (2 c^2 n)}. $$
\end{proposition}

\begin{proof}[Proof of Lemma~\ref{lemma-switchset}:]
Consider a pseudo-tree $H \in N_{\mathcal{H}}(v|t)$, and let $u\in V(H)\setminus \{v\}$. Since ${\rm conf}(\mathcal{H}'|t) \leq \delta$, there are no more than $\delta d_{\mathcal{H}'}(v|t)$ trees in $N_{\mathcal{H}'}(v|t)$ containing $u$. Thus, there are at least $(1-\delta m)d_{\mathcal{H}'}(v|t)$ pseudo-copies of $T$ in $N_{\mathcal{H}'}(v|t)$ that intersect $H$ only in $v$. Since $d_{\mathcal{H}}(v|t)\leq \varepsilon d_{\mathcal{H}'}(v|t)$, we can associate a set $S(H)$ of $\lfloor\frac{1-\delta m}{\varepsilon}\rfloor$ pseudo-copies in $N_{\mathcal{H}'}(v|t)$ with each $H \in N_{\mathcal{H}}(v|t)$ such that each element of $N_{\mathcal{H}'}(v|t)$ is contained in at most one of these sets. We define the function $f_{t}$ by choosing $f_t(H)$ uniformly at random from one of the pseudo-trees in $S(H)$. Clearly, any such function will satisfy the first two conditions of Lemma~\ref{lemma-switchset}. All that remains to show is that with positive probability $d_{f_t(\mathcal{H})}(v|t') \leq 3\varepsilon d_{\mathcal{H}'}(v|t')$ holds for every compatible $v\in V(G)$, $t'\in V(T)$.

The value of $d_{f_t(\mathcal{H})}(v|t')$ only depends on the set of pseudo-trees in $N_{\mathcal{H}'}(v|t')$
that are contained in some $S(H)$. Let $\mathcal{H}''$ be the collection of pseudo-copies of $\mathcal{H}'$ that are contained in some $S(H)$.
Clearly, each tree in 
$N_{\mathcal{H}''}(v|t')$ can be matched with exactly one tree in $\mathcal{H}$ and this occurs
with probability $\lfloor\frac{1-\delta m}{\varepsilon}\rfloor^{-1}$. By linearity of expectation,  
$$\mathbb{E}[d_{f_t(\mathcal{H})}(v|t')] =  \left\lfloor\frac{1-\delta m}{\varepsilon}\right\rfloor^{-1} d_{\mathcal{H}''}(v|t')<2\varepsilon  d_{\mathcal{H}'}(v|t')\,.$$ 

Let $A_{v, t'}$ be the event that $d_{f_t(\mathcal{H})}(v|t') > 
3\varepsilon  d_{\mathcal{H}'}(v|t')$. 
Note that $d_{f_t(\mathcal{H})}(v|t')$ is completely determined
by at most $d_{\mathcal{H}''}(v|t')$ independent trials. Since the outcome of each trial can affect $d_{f_t(\mathcal{H})}(v|t')$ by at most 1, the Simple Concentration Bound gives
\begin{align*}
\PP[A_{v,t'}] & < 2e^{- \varepsilon^2 d_{\mathcal{H}''}(v|t')/2}.
\end{align*}

We claim that $A_{v, t'}$ is mutually independent of all but at most $m\lfloor\frac{1-\delta m}{\varepsilon}\rfloor d_{\mathcal{H}''}(v|t')$ other events $A_{v',t''}$. 
Indeed, $A_{v, t'}$ depends on at most $d_{\mathcal{H}''}(v|t')$ random trials, and in each trial we have a choice of $\lfloor\frac{1-\delta m}{\varepsilon}\rfloor$
trees to match. Each tree affects precisely $m$ events other than $A_{v,t'}$.

Now we apply Proposition~\ref{prop: Lovasz} to show that with positive probability none of the events $A_{v, t'}$ occur. Set $x= \frac{\varepsilon^4}{8}$. It is sufficient to show that 
$$x\left(1 - x\right)^{m\lfloor\frac{1-\delta m}{\varepsilon}\rfloor d_{\mathcal{H}''}(v|t')} \geq  \PP[A_{v, t'}]$$
holds for all compatible $v\in V(G)$, $t'\in V(T)$. If $d_{\mathcal{H}''}(v|t') < \left(\frac{2}{\varepsilon}\right)^6$, then $d_{f_t(\mathcal{H})}(v|t') < \left(\frac{2}{\varepsilon}\right)^6 < 3\varepsilon  d_{\mathcal{H}'}(v|t')$, so $\PP[A_{v, t'}]=0$. If $d_{\mathcal{H}''}(v|t') \geq \left(\frac{2}{\varepsilon}\right)^6$, then we have

\begin{eqnarray*}
x\left(1 - x\right)^{m\lfloor\frac{1-\delta m}{\varepsilon}\rfloor d_{\mathcal{H}''}(v|t')} & \geq & x\left(1 - x\right)^{d_{\mathcal{H}''}(v|t')/\varepsilon^2}\\
& \geq & xe^{-2x d_{\mathcal{H}''}(v|t')/\varepsilon^2}\\
& \geq & \PP[A_{v, t'}] \cdot \frac{x}{2}\cdot e^{\frac{\varepsilon^2}{4}d_{\mathcal{H}''}(v|t')}\\
& \geq & \PP[A_{v, t'}] \cdot \left(\frac{\varepsilon}{2}\right)^6 d_{\mathcal{H}''}(v|t')\\
& \geq & \PP[A_{v, t'}]\,.
\end{eqnarray*}

By the Local Lemma, there is a positive probability that none of the bad events occur.
Thus, there exists a function $f_t$ with the desired properties.
\end{proof}

We now have all ingredients for the proof of Theorem~\ref{theorem:main}. 

\begin{proof}[Proof of Theorem~\ref{theorem:main}:]
As described in Step 2 of Section 2, let $t_0,\ldots ,t_m$ be a labelling of the vertices of $T$ such that $ T[t_0,\ldots ,t_i]$  is connected for every $i\in \{1,\ldots m\}$.
We also label the edges of $T$ so that $e_i$ denotes the edge joining $t_i$ with $T[t_0,\ldots ,t_{i-1}]$ for every $i\in \{1,\ldots m\}$. 
Set $\varepsilon_i = 5^{i-m}/15m$ for $i\in \{1,\ldots ,m\}$. We are going to construct a sequence $(\mathcal{H}_{i} \cup \mathcal{I}_{i})_{i=1}^{m}$ of $T$-pseudo-decompositions of $G$ such that the following holds:
\begin{itemize}
\item $\mathcal{I}_i$ is a collection of isomorphic copies of $T$ for every $i\in \{1,\ldots m\}$; 
\item $\mathcal{H}_i$ is $i$-good for every $i\in \{1,\ldots m\}$;
\item $d_{\mathcal{I}_i}(v|t) > {\rm max}\{22/\varepsilon_i^{7}, d_{\mathcal{H}_i}(v|t)/\varepsilon_i \} $ for every compatible $v\in V(G)$, $t\in V(T)$;
\item ${\rm conf}(\mathcal{I}_i) \leq \varepsilon_i$ for every $i\in \{1,\ldots ,m\}$.
\end{itemize}

Since the minimum degree in each colour in $G$ is at least $10^{50m}$, we can apply Lemma~\ref{dense} with parameters $\varepsilon = \delta = 10^{-2m}$.
Let $\mathcal{H} \cup \mathcal{I}$ denote the resulting $T$-pseudo-decomposition.
Clearly $\mathcal{H} \cup \mathcal{I}$ satisfies the conditions for $\mathcal{H}_1\cup \mathcal{I}_1$. 
Let $i\in \{2,\ldots ,m\}$ and suppose we have constructed $\mathcal{H}_{i-1}\cup \mathcal{I}_{i-1}$ such that the conditions above are satisfied. 
We need to repair the pseudo-trees in $\mathcal{H}_{i-1}$ that are not $i$-good. 
Since the pseudo-trees in $\mathcal{H}_{i-1}$ are all $(i-1)$-good, we can achieve this by making $i$-switches. 
Let $t_j$ be the endpoint of $e_i$ that is different from $t_i$. Let $f_j:\mathcal{H}_{i-1}\rightarrow \mathcal{I}_{i-1}$ be the function we get by applying Lemma~\ref{lemma-switchset} with $\mathcal{H}=\mathcal{H}_{i-1}$, $\mathcal{H'}=\mathcal{I}_{i-1}$, $\varepsilon =\delta = \varepsilon_{i-1}$, and $t=t_{j}$.
Now $f_j(\mathcal{H}_{i-1})$ is the set of trees we use to repair the pseudo-trees in $\mathcal{H}_{i-1}$ that are not $i$-good. 
Set 
\begin{eqnarray*}
\mathcal{H}_i & = & \bigcup_{H\in \mathcal{H}_{i-1}} {\rm sw}_i(H,f_j(H))\,\mbox{ and} \\
\mathcal{I}_i & = & \mathcal{I}_{i-1}\setminus f_j(\mathcal{H}_{i-1})\,,
\end{eqnarray*}
where ${\rm sw}_i(H,f_j(H))$ denotes the $i$-switch of $H$ and $f_j(H)$ as defined in Section 2.
Since $H \cap f_j(H) = \{v\}$ for every $H\in N_{\mathcal{H}_{i-1}}(v|t_{j})$, the two pseudo-copies in ${\rm sw}_i(H,f_j(H))$ are both $i$-good. 

Notice that the degree $d_{\mathcal{H}_i}(v|t)$ of a vertex is invariant under $i$-switches between pseudo-trees in $\mathcal{H}_i$.
Since $d_{f_j(\mathcal{H}_{i-1})}(v|t) \leq 3\varepsilon_{i-1} d_{\mathcal{I}_{i-1}}(v|t)$ holds for compatible $v\in V(G)$ and $t\in V(T)$, we have 
$d_{\mathcal{I}_{i}}(v|t)\geq (1-3\varepsilon_{i-1})d_{\mathcal{I}_{i-1}}(v|t)$
and
$d_{\mathcal{H}_{i}}(v|t)\leq 4\varepsilon_{i-1} d_{\mathcal{I}_{i-1}}(v|t)\,.$
Thus, 
$$d_{\mathcal{H}_{i}}(v|t) \leq \frac{4\varepsilon_{i-1}}{(1-3\varepsilon_{i-1})}d_{\mathcal{I}_{i}}(v|t) \leq 5\varepsilon_{i-1} d_{\mathcal{I}_{i}}(v|t) =\varepsilon_i d_{\mathcal{I}_{i}}(v|t)\,,$$
$$d_{\mathcal{I}_{i}}(v|t) \geq (1-3\varepsilon_{i-1})d_{\mathcal{I}_{i-1}}(v|t) \geq 22\frac{1-3\varepsilon_{i-1}}{\varepsilon_{i-1}^{7}} \geq \frac{22}{\varepsilon_{i}^7}$$
and
\begin{eqnarray*}
{\rm conf}(\mathcal{I}_i) & \leq & \frac{{\rm conf}(\mathcal{I}_{i-1})}{1-3\varepsilon_{i-1}} \leq \frac{5}{4}\varepsilon_{i-1} < \varepsilon_i\,.
\end{eqnarray*}
Hence, the $T$-pseudo-decomposition $\mathcal{H}_i\cup \mathcal{I}_i$ has the desired properties. In particular, $\mathcal{H}_m$ is $m$-good and $\mathcal{H}_{m} \cup \mathcal{I}_{m}$ is a $T$-decomposition of $G$.
\qedhere
\end{proof}


\begin{thebibliography}{99}

\bibitem{BG14} J. Bar\'at and D. Gerbner. Edge-Decomposition of Graphs into Copies of a Tree with Four Edges. \textit{Electronic Journal of Combinatorics}, 21(1):\#P1.55, 2014.

\bibitem{BKLO16} B. Barber, D. K\"uhn, A. Lo, and D. Osthus. Edge-decompositions of graphs with high minimum degree. \textit{Advances in Mathematics}, 288: 337-385, 2016.

\bibitem{BT06} J. Bar\'at and C. Thomassen. Claw-decompositions and Tutte-orientations. \textit{Journal of Graph Theory}, 52:135-146, 2006.

\bibitem{BHLT15+} J. Bensmail, A. Harutyunyan, T.-N. Le, and S. Thomass\'e. Edge-partitioning a graph into paths: beyond the Bar\'at-Thomassen conjecture. arXiv:1507.08208 .

\bibitem{BMOW14} F. Botler, G.O. Mota, M. Oshiro, and Y. Wakabayashi. Decomposing highly edge-connected graphs into paths of any given length. \textit{Journal of Combinatorial Theory, Series B}, in press, http://dx.doi.org/10.1016/j.jctb.2016.07.010.

\bibitem{BMOW15+} F. Botler, G.O. Mota, M. Oshiro, and Y. Wakabayashi. Decomposing highly connected graphs into paths of length five. \textit{Discrete Applied Mathematics}, in press, http://dx.doi.org/10.1016/j.dam.2016.08.001.

\bibitem{DT97} D. Dor and M. Tarsi. Graph Decomposition is NP-Complete: A Complete Proof of Holyer's Conjecture. \textit{SIAM Journal on Computing}, 26(4):1166-1187, 1997.

\bibitem{LTWZ13} L.M. Lov{\'{a}}sz, C. Thomassen, Y. Wu, and C.-Q. Zhang. Nowhere-zero 3-flows and modulo $k$-orientations. \textit{Journal of Combinatorial Theory, Series B}, 103:587-598, 2013.

\bibitem{M02} C. McDiarmid. Concentration for Independent Permutations.
\textit{Combinatorics, Probability and Computing}, 11:163-178, 2002.

\bibitem{Mer15+} M. Merker. Decomposing highly edge-connected graphs into homomorphic copies of a fixed tree. \textit{Journal of Combinatorial Theory, Series B}, in press, http://dx.doi.org/10.1016/j.jctb.2016.05.005.

\bibitem{MR02} M. Molloy and B. Reed. \textit{Graph Colouring and the Probabilistic Method}. Springer, 2002.

\bibitem{Tha08} C. Thomassen. Decompositions of highly connected graphs into paths of length $3$. \textit{Journal of Graph Theory}, 58:286-292, 2008.

\bibitem{Thb08} C. Thomassen. Edge-decompositions of highly connected graphs. \textit{Abhandlungen aus dem Mathematischen Seminar der Universit\"at Hamburg}, 18:17-26, 2008.

\bibitem{Tho12} C. Thomassen. The weak $3$-flow conjecture and the weak circular flow conjecture. \textit{Journal of Combinatorial Theory, Series B}, 102:521-529, 2012.

\bibitem{Tha13} C. Thomassen. Decomposing graphs into paths of fixed length. \textit{Combinatorica}, 33(1):97-123, 2013.

\bibitem{Thb13} C. Thomassen. Decomposing a graph into bistars. \textit{Journal of Combinatorial Theory, Series B}, 103:504-508, 2013.

\bibitem{W76} M. Wilson. Decompositions of complete graphs into subgraphs isomorphic to a given graph.
In \textit{Proceedings of the Fifth British Combinatorial Conference (Univ. Aberdeen, Aberdeen, 1975)}, pages 647 - 659. Congressus Numerantium, No. XV, Utilitas Math., Winnipeg, Man., 1976.

\end{thebibliography}
\end{document}